\documentclass{article}

\usepackage{amsthm,amsmath,amsfonts,amssymb,stmaryrd,euscript,authblk,dsfont,relsize,scalerel,mathtools}
\usepackage{hyperref}       
\usepackage{url}            
\usepackage{nicefrac}       
\usepackage{microtype}      
\usepackage{xcolor}         
\usepackage[normalem]{ulem}

\newcommand{\E}{\mathbb{E}}

\newcommand{\A}{\mathcal{A}}
\newcommand{\B}{\mathcal{B}}

\newcommand{\e}{\operatorname{e}}
\newcommand{\N}{\ensuremath{\mathds{N}}}    
\newcommand{\R}{\ensuremath{\mathds{R}}}     
\newcommand{\Z}{\ensuremath{\mathds{Z}}}    
\renewcommand{\P}{\ensuremath{\mathds{P}}} 
\newcommand{\Q}{\ensuremath{\mathds{Q}}} 

\newcommand{\sP}{\scaleto{\P}{4pt}}

\DeclareMathOperator{\Var}{Var}

\newcommand{\comp}{\mathrm{c}}
\DeclareMathOperator{\dd}{\textup{d}\!}
\newcommand{\1}{\mathds{1}}

\newcommand{\KL}{\mathrm{KL}}
\DeclareMathOperator*{\argmax}{arg\,max}
\DeclareMathOperator*{\argmin}{arg\,min}
\let\epsilon=\varepsilon
\let\geq=\geqslant
\let\leq=\leqslant

\newcommand{\smallDG}{\scaleto{D,G}{5pt}}

\newtheoremstyle{mythmstyle}
  {15pt}
  {15pt}
  { }
  { }
  {\bfseries\scshape}
  {.}
  { }
  { }

\theoremstyle{mythmstyle}
\newtheorem{thm}{Theorem}[section]
\newtheorem{lem}{Lemma}[section]
\newtheorem{prop}{Proposition}[section]
\newtheorem{coro}{Corollary}[section]
\newtheorem{defn}{Definition}[section]

\newtheorem{assumption}{Assumption}

\makeatletter
\renewenvironment{proof}[1][\bfseries\proofname]{\par
  \pushQED{\qed}%
  \normalfont \topsep6\p@\@plus6\p@\relax
  \trivlist
  \itemindent\z@ 
  \item[\hskip\labelsep
        \scshape
    #1\@addpunct{.}]\ignorespaces
}{%
  \popQED\endtrivlist\@endpefalse
}
\makeatother

\definecolor{monvert}{rgb}{0.0,.5,0.0}
\definecolor{britishracinggreen}{rgb}{0.0, 0.26, 0.15}
\definecolor{monbleu}{rgb}{0,.2,.8}
\definecolor{applegreen}{rgb}{0.55, 0.71, 0.0}
\definecolor{monrouge}{rgb}{0.8, 0.0, 0.0} 
\definecolor{green(ncs)}{rgb}{0.0, 0.62, 0.42}
\definecolor{indiagreen}{rgb}{0.07, 0.53, 0.03}
\definecolor{jade}{rgb}{0.0, 0.66, 0.42}
\definecolor{mediumseagreen}{rgb}{0.24, 0.7, 0.44}
\definecolor{tealgreen}{rgb}{0.0, 0.51, 0.5}
\definecolor{tealblue}{rgb}{0.21, 0.46, 0.53}
\definecolor{warmblack}{rgb}{0.0, 0.26, 0.26}
\definecolor{yaleblue}{rgb}{0.06, 0.3, 0.57}
\definecolor{prussianblue}{rgb}{0.0, 0.19, 0.33}
\definecolor{cadmiumgreen}{rgb}{0.0, 0.42, 0.24}
\definecolor{royalblue(traditional)}{rgb}{0.0, 0.14, 0.4}
\definecolor{midnightblue}{rgb}{0.1, 0.1, 0.44}

\begin{document}

\title{Pathwise guessing in categorical time series\\ with unbounded alphabets}

\author[1]{J.-R. Chazottes 
\thanks{Email: jeanrene@cpht.polytechnique.fr}}

\author[2]{S. Gallo
\thanks{Email: sandro.gallo@ufscar.br}}

\author[3]{D. Y. Takahashi
\thanks{Email: takahashiyd@gmail.com}}

\affil[1]{Centre de Physique Th\'eorique, CNRS, \'Ecole Polytechnique, Institut Polytechnique de Paris, France}

\affil[2]{Departamento de Estat\'istica, Universidade Federal de S\~ao Carlos, S\~ao Paulo, Brazil}

\affil[3]{Instituto do C\'erebro, Universidade Federal do Rio Grande do Norte, Natal, Brazil}

\maketitle

\begin{abstract}
The following learning problem arises naturally in various applications: Given a finite sample from a categorical or count time series, can we learn a function of the sample that (nearly) maximizes the probability of correctly guessing the values of a given portion of the data using the values from the remaining parts? Unlike classical approaches in statistical inference, our approach avoids explicitly estimating the conditional probabilities.

We propose a non-parametric guessing function with a learning rate independent of the alphabet size. Our analysis focuses on a broad class of time series models that encompasses finite-order Markov chains, some hidden Markov chains, Poisson regression for count processes, and one-dimensional Gibbs measures. 

We provide a margin condition that controls the rate of convergence for the risk. Additionally, we establish a minimax lower bound for the convergence rate of the risk associated with our guessing problem. This lower bound matches the upper bound achieved by our estimator up to a logarithmic factor, demonstrating its near-optimality. 

\bigskip

\end{abstract}



\section{Introduction}

Prediction and interpolation, \emph{i.e.}, \emph{guessing}, are foundational problems in science with wide-ranging applications, particularly in categorical time series. Prediction involves estimating future values of a sequence based on prior observations, while interpolation focuses on inferring missing values within a sequence of known data. Both tasks aim to construct a function from observed data that can accurately predict or interpolate the missing components. These problems have garnered renewed interest due to the rise of large language models, which can be viewed as categorical time series models with large alphabets \cite{gruver2024large}. In this new context, both the alphabet size and the dependence order are typically high. Our work specifically addresses models with unbounded dependence and alphabet sizes, making them suitable for modern applications.

A common approach to these problems relies on estimating conditional probabilities. Classical prediction involves pointwise estimation of all transition probabilities, as discussed in \cite{morvai-weiss/2021}. 
Other studies, such as \cite{wolfer2019minimax} and \cite{hao2018learning}, use PAC learning frameworks to investigate optimal rates for learning conditional probabilities across various metrics, with \cite{han2023optimal} addressing $k$-step Markov chains with large alphabets using minimax Kullback-Leibler risk. This last article introduces alternative loss functions to mitigate sensitivity to rare transitions, showing that optimal prediction can be achieved with an alphabet size scaling as $\mathcal{O}(\sqrt{n}\,)$. Furthermore, parametric regression models, known for their flexibility and algorithmic efficiency \cite{fokianos1996categorical}, are another common approach to time series prediction. However, such models often impose restrictive assumptions on the underlying processes and may not efficiently solve the guessing problem, in general. In all these approaches, the guessing becomes challenging when the alphabet size is large or when transition probabilities are very small. Accurately estimating rare events requires a substantial amount of data, which can be impractical. 

A more practical approach is to focus on the events that are most likely to occur. This means focusing our efforts on predicting the most probable outcomes and giving less weight to rare, unlikely events. This approach is especially useful when dealing with sequences with a large or infinite number of possible symbols. Even with large sets of symbols, we expect that such a method can still make accurate predictions, unlike traditional methods that might struggle to estimate the probabilities of all possible transitions.

This paper presents a probabilistic guessing scheme that addresses these challenges. In guessing, we want to learn from the sample a function that correctly guesses the values on a given portion of the time series given the remaining parts. Our non-parametric estimator applies to a broad class of categorical time series \cite{kedem2005regression}, accommodating arbitrary alphabet sizes and relaxing constraints on memory or order. A key feature of our approach is that the learning rate of the associated risk is independent of the size of the alphabet, and we establish that this rate is essentially sharp. To the best of our knowledge, our work is the first to consider the infinite alphabet case, which is relevant when there is no prior upper bound on the alphabet size or the alphabet size can increase with the sample size.
The proof of the upper bound is based on a Dvoretzky-Kiefer-Wolfowitz type inequality, which we previously established in a slightly different form \cite{chazottes2023gaussian}. To obtain a sharp control on the convergence rate of the risk, we introduce a margin-type condition. In particular, when the margin decreases fast with the sample size $n$, we show that the risk converges as $\mathcal{O}(\sqrt{\log/n})$. On the other hand, when the margin is bounded away from zero, we prove an upper bound for the risk that is exponential in the sample size $n$. To prove the lower bound, we use the classical Le Cam type of argument \cite{yu1997cam}.

This paper is organized as follows. Section \ref{sec:guessing_scheme} introduces the basic notation and formalizes the probabilistic guessing scheme. Section \ref{sec:results} details our assumptions and presents the main theoretical results. Examples of time series models that satisfy our assumptions are discussed in Section \ref{sec:examples}, and the proofs of the main results are provided in Section \ref{sec:proofs}. 

\section{The probabilistic guessing problem}\label{sec:guessing_scheme}

We will now explain the formulation of the probabilistic guessing problem. Before that, let us introduce the notation.

\paragraph{Notation.} Let $\A$ be a countable set with cardinality at least $2$ (the set of categories, or the alphabet) endowed with the discrete topology and denote by $\A^{\Z}$  the set of bi-infinite sequences drawn from $\A$. We then put the product topology on $\A^\Z$. It is generated by the cylinder sets $[a_{-n+1},\ldots,a_{n-1}]=\{x\in \A^\Z : x_i=a_i, |i|\leq n-1\}$, $a_i\in \A$, $n\in\N$, and comprises all Borel sets of $\A^\Z$. 

We use the notation $\Subset$ to mean ``a finite subset of''.
Given $H \Subset \Z$, the diameter of $H$ is defined as $\text{diam}(H) := \max H - \min H + 1$.
For integers $m<n$, let $\llbracket m,n\rrbracket:=\left[m,n\right]\cap \Z$.

For any pair $x,y\in\A^\Z$, and any $\Lambda\Subset\Z$, the point $z=x_\Lambda y_{\Lambda^{\!\comp}}$ is defined by $z_i=x_i$ for $i\in\Lambda$, and $z_i=y_i$ for $i\notin \Lambda$.
We define the shift operator $\theta$ acting on subsets of $\Z$. Namely, for $H \subset \Z$ and $i\in\Z$, let $\theta^i H = H+i$, where $H+i=\{i+j:j\in H\}$. Similarly,  the shift operator $T$ on $\A^\Z$ is
defined by $(T^ix)_n=x_{n+i}$, for $i,n \in \Z$. 

Given a $\A$-valued process $(X_i)_{i\in\Z}$, we write $X_m^n=(X_m,X_{m+1},\ldots,X_n)$ ($m<n$), and more generally, for $G\subset \Z$, $X_G=(X_i)_{i\in G}$. We denote by $\sigma(X_G)$ the $\sigma$-algebra generated by $(X_i)_{i\in G}$. 

\paragraph{The probabilistic guessing problem.}

We want to find an estimator in which the probability of correctly guessing a given subset of symbols is close (with a precision we can choose) to the best possible guess. 
We will formulate this problem as follows. Let $(X_j)_{j \in \Z}$ and $(Y_j)_{j \in \Z}$ be two independent  copies of the same process with distribution $\P$. We want to use the information of the {\em data set} $D$ to predict the value on the {\em guess set} $G$. An estimator is a function 
\begin{align*}
\hat{f}^n_{\smallDG}:\A^n \times \A^D 
&\to \A^G\\
\hat{f}^n_{\smallDG}[x^n_1]
(y_{\scaleto{D}{4.5pt}}) 
&= y_{\scaleto{G}{4.5pt}}\,.
\end{align*}
Assuming that we observe a finite sample $X_1, \ldots, X_n$, we think of it as a random function, namely,
\[
y_{\scaleto{D}{4.5pt}}\mapsto\hat{f}^n_{\smallDG}[X^n_1](y_{\scaleto{D}{4.5pt}}):=\hat{f}^n_{\smallDG}(y_{\scaleto{D}{4.5pt}})\,.
\] 
Two natural cases are the ``prediction'' problem, which corresponds to the case $\max D<\min G$,
and the ``interpolation''  problem, which corresponds to the case $G\subset \llbracket\,\min D,\max D\,\rrbracket$.

To measure the performance of the estimator $\hat{f}^n_{\smallDG}$, we define the {\em excess risk} as  
\begin{align} \label{eq:risk}
& R(\hat{f}^n_{\smallDG}) \\
&:= \sup_{b\, \in \A^D}\!\Big(\widetilde{\P}\big(\hat{f}^n_{\smallDG}(Y_D) \neq  Y_G \mid Y_D = b\big) 
- \!\inf_{a\, \in \A^G}\widetilde{\P}\big(a \neq  Y_G \mid Y_D = b\big)\Big)\,\widetilde{\P}(Y_D = b), \notag
\end{align}
where $\widetilde{\P} = \P \otimes \P$ (product measure).
Note that $R(\hat{f}^n_{\smallDG})$ depends on $\P$, although we do not explicitly indicate this in the notation as long as $\P$ is 
fixed. The above risk is designed to uniformly control the probability of making the correct guess for each data
$\{Y_D = b\}$ weighted by the probability that we observe such data.
 
We will discuss some alternative risks to better understand the choice of our risk in \eqref{eq:risk}. An alternative risk could be the ``unweighted'' risk:
\begin{align*} 
\sup_{b\, \in \A^D} \widetilde{\P}\big(\hat{f}^n_{\smallDG}(Y_D) \neq Y_G \mid Y_D = b\big) 
- \inf_{a\, \in \A^G} \widetilde{\P}\big(a \neq Y_G \mid Y_D = b\big)\,. \notag
\end{align*} 
The drawback of this alternative is that it requires controlling errors even when $\widetilde{\P}(Y_D = b)$ is small. This typically requires larger sample sizes with minimal improvement in the overall probability of guessing.

Another natural risk can be defined using the function $f: \A^D \to \A^G$, where
\[ 
f(b) = \argmin_{a\, \in \A^G} \P\big(a \neq Y_G \mid Y_D = b\big)\,.
\] 
The minimum is guaranteed to exist since $\P$ is a probability measure. A seemingly equivalent risk to \eqref{eq:risk} is then given by
\begin{align*} 
\sup_{b\, \in \A^D} \widetilde{\P}\big(\hat{f}^n_{\smallDG}(Y_D) \neq f(Y_D), Y_D = b\big)\,.
\end{align*}
However, this choice has its own caveat. When several symbols have probabilities close to the minimum, estimating $f$ becomes more challenging. Despite this, these symbols yield similar guessing errors, making it unnecessary to distinguish between them.

Consider a simple example: a biased die where the side labeled with the symbol ``one'' has the highest probability of 
appearing. If each roll is independent, the optimal prediction is to always guess ``one''. In this scenario, it suffices 
to correctly identify that ``one'' has the highest probability; there is no need to precisely estimate the probabilities 
of the other symbols. Furthermore, the total number of symbols on the die is irrelevant to the accuracy of this 
prediction.

Now, suppose the probabilities for all symbols on the die are nearly identical. In this case, any symbol can be chosen, and the probability of a correct guess will still be close to the theoretical optimum. Here, even though it might be challenging to identify the symbol with the highest probability, the exact identification might be less relevant, as the prediction error 
is negligible. Such situations are common when the number of categories is large.

Finally, consider a die with a Markovian dynamics, where the outcome of each roll depends on the previous one. In this 
case, the optimal guess depends on the past observations, with the accuracy of the guess tied to the probabilities 
associated with prior outcomes. When a past outcome has a low probability, it becomes less critical to make an accurate 
prediction because the event requiring a guess will occur rarely. Moreover, even when rare events happen, the chance of correctly predicting them remains slim. Our proposed formulation explicitly captures and formalizes this idea.

To conclude with our formulation of the guessing problem, we observe that it shares similarities with non-parametric multiclass classification \cite{del2022multiclass}. However, a key distinction is that we do not aim to classify all labels accurately. Instead, we focus only on those labels with a high probability of occurrence, enabling us to extend the method to cases involving countably infinite 
alphabets. Moreover, for multiclass classification, the samples are usually i.i.d., which is not our case. 

\paragraph{The estimator.}

We will use the following estimator.

\begin{defn}\label{def:estimator}
Given $n\geq1$, $X^n_1$, $G,D\Subset \Z$, such that $\min (G \cup D) =1$, and $b \in \A^{D}$, let
\begin{align}\label{eq:estimator}
\hat{f}_{\smallDG}^n[X_1^n](b):= \argmax_{a\, \in\, \A^G} \frac{N^n_{\smallDG}[X_1^n](b,a)}{N^n_{\smallDG}[X_1^n](b)} = \argmax_{a\, \in\, \A^G} \frac{N^n_{\smallDG}[X_1^n](b,a)}{n}\,,
\end{align}
where
\[
N^n_{\smallDG}[X_1^n](b,a): = \sum_{i = 0}^{n-1} \1 \{X_{\theta^i D} = b, X_{\theta^i G} = a\}
\]
and
\[
N^n_{\smallDG}[X_1^n](b): = \sum_{a \in \A^G} N^n_{\smallDG}[X_1^n](b,a)
\]
are the counting functions, with the convention that the indicator function $\1\{\}$ is zero if $\theta^iD$ or $\theta^i G$ is not included in $\llbracket 1,n\rrbracket$. We will often omit the reference to the sample on which we count occurrences, writing $\hat{f}_{\smallDG}^n(b)$ and $N^n_{\smallDG}(b,a)$. The key observation for our estimator is that the maximizing argument for the conditioned empirical probability is equal to the maximizing argument for the unconditioned empirical probability.
\end{defn}

\section{Assumption and statement of the results} \label{sec:results}

Our main assumption about the categorical process concerns the degree of dependence it exhibits on its ``past''.
Let us denote by $p$ a regular version of the left conditional expectations of the process $(X_i)_{i\in\Z}$ with 
distribution $\P$, that is, $p:\A\times \A^{\Z_-}\rightarrow[0,1]$ is such that
\begin{equation}\label{eq:conditionaldef}
\E\Big({\1}_{\{X_0=a\}}\,\big\vert\,\sigma\big(X_{-\infty}^{-1}\big)\Big)(x)=p(a|x),\;\;\P\text{-a.s.},
\end{equation}
where $\Z_-:=\{-1, -2, \ldots\}$ is the set of negative integers.
We say that, for $\P$-almost all ``pasts'' $x\in \A^{\Z_-}$, $p$ specifies the transition probabilities to any symbol of $\A$. Let us also define the \emph{variation of $p$}, $(\text{Var}_n(p))_{n\ge0}$, as 
\[
\text{Var}_n(p):=
\sup\left\{\scaleto{\frac{1}{2}}{18pt}\sum_{a\,\in\A} \big|\,p(a|x)-p(a|y)\big|:x,y\in \A^{\Z_-}, 
x_i=y_i,i\ge-n\right\}\,,n\geq 0\,,
\]
with the convention that for $n=0$ the supremum is taken over any pair $x,y\in \A^{\Z_-}$ without restriction. 
Our main regularity assumption concerning $p$ is that $\text{Var}_n(p)$ goes to $0$ sufficiently fast when $n$ diverges. 
\begin{assumption}\label{eq:Gammadef}
Let $(X_i)_{i\in\Z}$ be a stationary process with measure $\P$. We say that $(X_i)_{i\in\Z}$ satisfies Assumption A if 
\[
\Gamma(\P) := \prod_{j=0}^\infty\big(1-\Var_j(p)\big)>0.
\]
\end{assumption}
Observe that Assumption \ref{eq:Gammadef} is equivalent to having 
\[
\Var_0(p) < 1\quad \text{and}\quad \sum_{j \geq 1} \Var_j(p) < \infty.
\]

Stationary processes defined by \eqref{eq:conditionaldef} are commonly referred to in the literature as chains of infinite order, chains with complete connections, stochastic chains with unbounded memory, or $g$-measures (in the context of ergodic theory); see {\em e.g.} \cite{fernandez/maillard/2005}. An insightful reference that examines these objects in the context of categorical time series is \cite{fokianos2019categorical}. 
Assumption \ref{eq:Gammadef} aligns with classical conditions from the literature, providing a framework to ensure that 
conditional expectations are not overly influenced by distant past events. This assumption captures a wide range of well-known models, including Markov chains and generalized linear models for categorical and count time series, one-dimensional Gibbs measures, as detailed in Section \ref{sec:examples}.

We now introduce a quantity that will play a key role in controlling the convergence rate of the risk.
Given $D,G\Subset \Z$, for each $b \in \A^D$, let 
\begin{align*}
& \scaleto{\delta}{8pt}_{\smallDG}(b) = \\
&\quad \inf\big\{ \P(X_G \neq c, X_{D} = b) - \inf_{a \in \A^G}\P(X_G \neq a, X_{D} = b) > 0: c \in \A^G \big\}.
\end{align*}
If the set inside the infimum is empty, we take  $\scaleto{\delta}{8pt}_{\smallDG}(b) = 0$. We define the \emph{margin}
\[
\delta_{\smallDG} := \inf_{b\, \in \A^D} \delta_{\smallDG}(b)\,,
\]
which quantifies the difficulty of the guessing problem, analogous to the margin condition in classification problems. 

Next, we introduce a quantity that controls the convergence rate of risk with respect to $D$ and $G$. Namely, let
\[
\beta_{\smallDG} = \sup_{b\, \in A^D}\bigg\{\sup_{a\, \in A^G}\P\big(Y_G =a, Y_D = b \big)-\inf_{c\, \in A^G}\P\big(Y_G =c, Y_D = b \big)\bigg\}.
\]
 To understand the relation between $\beta_{\smallDG}$ and the size of $D$ and $G$, we have the following result. 
 
 \begin{prop} \label{prop:expupperbound}
 If the left conditional probability $p$ of a process with distribution $\P$ is such that 
 \begin{equation} \label{eq:exponentialUpper}
 \bar{p} :=\sup_{x\, \in \A^{\Z_-}} \sup_{a\,\in \A} \, p(a|x) < 1,
 \end{equation}
 then 
 \begin{equation*}
 \beta_{\smallDG} \leq \bar{p}^{\,(|D|+|G|)}.
 \end{equation*}
 \end{prop}
 Hence, condition \eqref{eq:exponentialUpper} implies that $\beta_{\smallDG}$ decreases exponentially with $|D|+|G|$. We will see in Section \ref{sec:examples} that condition \eqref{eq:exponentialUpper} is satisfied for large classes of models.

We are now ready to state our results. 

\begin{thm} \label{theo:probguess}
Let $G,D\Subset \Z$, and put
$L_{\smallDG} :=\mathrm{diam}(G\cup D)$,
$K_{\smallDG}:=|G|+|D|$. Let $(X_j)_{j \in \Z}$ be a stationary process with measure $\P$ satisfying Assumption  \ref{eq:Gammadef} and margin $\delta_{\smallDG}$.
Let $\epsilon > 0$. If 
\begin{equation} \label{eq:samplesize}
n\geq \frac{4}{(\frac{\epsilon}{2} \vee \delta_{\smallDG})^{2}}\Bigg(\!\!\left(\frac{K_{\smallDG}}{\Gamma(\P)}\right)^2\!\log\left(\frac{2\beta_{\smallDG}}{\epsilon}\!\right)+\frac{4K_{\smallDG}(1-\Gamma(\P))}{\Gamma(\P)}+2\!\Bigg)+L_{\smallDG}-2,
\end{equation}
then
\begin{equation}\label{Shosta}
R\big(\hat{f}^n_{\smallDG}\big) \leq \epsilon \wedge \beta_{\smallDG}\,.
\end{equation}
\end{thm}

A direct consequence is that the convergence rate of the risk depends on the relationship between $\epsilon/2$ and $\delta_{\smallDG}$, as well as the relationship between $\epsilon/2$ and $\beta_{\smallDG}$. In what follows, we will study some different regimes for the convergence of the risk. We consider a family of stationary processes $(\P_n)_{n \geq 1}$, which allows us to consider an increasing set of probability measures when the sample size $n$ increases. Hence, the margin $\delta_{\smallDG}$ may depend on $n$, and we introduce the shorthand
\begin{equation}\label{deltan}
\delta_n := \delta_{\smallDG}^{(n)}\,,\; n \ge 1\,.
\end{equation}

In the corollary below, the risk
$R(\hat{f}^n_{\smallDG})$ is with respect to the measure $\P_n$. 

\begin{coro} \label{coro:rate}
Fix $G,D\Subset \Z$. Let $(\P_n)_{n \geq 1}$ be a family of measures for stationary processes with margin $\delta_n$ (defined in \eqref{deltan}).  If $\Gamma(\P_n) \geq \Gamma >0$ for all $n \geq 2$ and if
\begin{equation} \label{eq:subcritical}
    \delta_{n} \sqrt{\frac{n}{\log n}} \rightarrow 0\,,
\end{equation}
then, there exists $n_0 \geq 2$ such that for all $n \geq n_0$,
\begin{equation*}
 R(\hat{f}^n_{\smallDG})  \leq \frac{1}{2}\sqrt{\frac{\log n}{ n }}\wedge\beta_{\smallDG}\,.
\end{equation*}
On the other hand, if   
\begin{equation} \label{eq:supercritical}
    \delta_{n} \sqrt{\frac{n}{\log n}} \rightarrow \infty 
\end{equation}
then, there exists $n_0' \geq 1$ such that for all $n \geq n_0'$,
\begin{equation*}
R(\hat{f}^n_{\smallDG})  \leq \exp\left(- \frac{\Gamma^2\,n\, \delta_n^2}{8(|G|+|D|)^2}\right) \wedge \beta_{\smallDG}\,.
\end{equation*}
\end{coro}

An important special case is when $\delta_n = 0$ for all $n$, in which case the risk converges as $\mathcal{O}\big(\sqrt{n^{-1}\log n}\,\big)$. For instance, if the alphabet is infinite and $D \neq \emptyset$, then $\delta_{n} = 0$. When the alphabet is finite, we can have $\delta_n > 0$. A relevant case is when $\delta_n \geq \Delta > 0$ for all $n$, in which case the risk converges as $\mathcal{O}(\e^{-Cn})$ for some positive $C$.

Now, we want to understand whether the learning rate achieved in this corollary is optimal. The result below establishes that, 
up to a logarithmic factor, our estimator attains the minimax optimal rate of convergence for the risk with respect to the sample size $n$. 

To state the result, we let $\EuScript{P}$ denote the set of distributions of all stationary stochastic processes on alphabet $\A$ whose left conditional expectation satisfies Assumption  \ref{eq:Gammadef}.
In this statement, we will explicitly write the dependence on $\P$ of the excess risk (see \eqref{eq:risk}).

\begin{thm} \label{theo:lowerbound}
Fix $G,D\Subset \Z$. For every positive integer $n$, let $\Psi_n$ be the set of estimators $\psi_n:\A^n \times \A^D \to \A^G$. 
If $\EuScript{P}\!_n$ is the set of probability measures that satisfy $\inf_{\P\!_n \in \EuScript{P}\!_n}\Gamma(\P\!_n) \geq \Gamma > 0$ and whose margins are bounded above by $\delta_n$ with
\begin{equation} 
    \delta_{n} \sqrt{\frac{n}{\log n}} \rightarrow 0\,,
\end{equation}
then for all $n \geq 2$,
\begin{equation}
\inf_{\psi_n \in \Psi\!_n}\, \sup_{\P \in \EuScript{P}\!_n} R(\psi_n;\P)  \geq {\frac{\e^{-1}}{\sqrt n }}\Big(\frac{1}{4}\Big)^{|G|+|D|}.
\end{equation}

On the other hand, if $\EuScript{Q}_n$ is the set of probability measures that satisfy $\inf_{\P\!_n \in \EuScript{Q}_n}\Gamma(\P\!_n) \geq \Gamma > 0$ and whose margins are bounded below by $\delta_n$ with
\begin{equation} 
    \delta_{n} \sqrt{\frac{n}{\log n}} \rightarrow \infty\,,
\end{equation}
then for all $n \geq 2$,
\begin{equation}
\inf_{\psi_n \in \Psi\!_n}\, \sup_{\P \in \EuScript{Q}_n} R(\psi_n;\P)  \geq \delta_n\e^{-n\delta_n^2}\left(\frac{1}{4}\right)^{|D|+|G|}.
\end{equation}
\end{thm}

It is interesting to observe that when the left conditional probabilities satisfy \eqref{eq:exponentialUpper}, from Proposition \ref{prop:expupperbound}, we have that the upper bound decreases exponentially in $|D|+|G|$. Given that the minimax lower bound obtained in Theorem \ref{theo:lowerbound} also decreases exponentially in $|D|+|G|$, our bound has the correct order with respect to the size of data and guessing sets.

\section{Examples satisfying Assumption \ref{eq:Gammadef}}\label{sec:examples}

As noted earlier, Assumption \ref{eq:Gammadef} is broadly applicable. To illustrate this, we present several well-studied models of categorical time series along with concrete examples that satisfy the assumption.
 
\paragraph{Independent random variables.}
For independent random variables, the function \( p \) defined in \eqref{eq:conditionaldef} is independent of \( x \in \A^{\Z_-} \). Consequently, \( \text{Var}_n(p) = 0 \) for all \( n \geq 0 \), which implies \( \Gamma(p) = 1 \). Even in this simple case, our result seems to be novel, as it provides a learning rate that does not depend on the alphabet size.  We observe that if condition \eqref{eq:exponentialUpper} is not satisfied, we would have $\sup_{a \in \mathcal{A}}\P(X_1 = a) = 1$, in which case the guessing problem becomes trivial. Hence, for all practical purposes, we only need to consider when condition \eqref{eq:exponentialUpper} is satisfied.

\paragraph{Markov chains.}
Consider a Markov chain with state space $\A$ and transition matrix $Q=(Q(a,b))_{a,b\in\A}$. In this case $p(a|x)=Q(x_{-1},a)$ for any $x\in \A^{\Z_-}$, thus $\text{Var}_n(p)=0$ for any $n\ge1$, and condition \eqref{eq:Gammadef} becomes equivalent to 
\[
\Var_0(p)=d(Q):=\sup_{a,b\,\in\A}\Vert\, Q(a,\cdot)-Q(b,\cdot)\Vert_{{\scriptscriptstyle \mathrm{TV}}}<1,
\]
and $\Gamma(p)$ can be substituted by $1-d(Q)$. The quantity $d(Q)$ is called the Dobrushin ergodicity coefficient (see, for example, \cite{dobrushin1956central} and \cite[Section 18.2]{douc2018markov}). 
More generally, a $k$ step Markov chain will satisfy $\text{Var}_n(p)=0$ for any $n\ge k$, and therefore our results also apply as long as $\Var_0(p)<1$.

\paragraph{Autoregressive models.}
Consider a function $\Upsilon: \R \to (0,1)$ such that $\Upsilon(r) + \Upsilon(-r) = 1$ and an absolutely summable sequence of real numbers
$(\xi_j)_{j \geq 0}$. Consider then the binary process $(X_j)_{j\in\Z}$ with transition probabilities given by
\[
p(a|x) = \Upsilon\bigg(a\sum_{j=1}^\infty \xi_j\, x_{-j} + a\xi_0\bigg).
\]
The process generated by this kernel is called a binary autoregressive process \cite{kedem2005regression}. A simple and yet classical situation is when $\Upsilon(u)=\big(1+\e^{-2u}\big)^{-1}$, $\xi_i \geq \xi_j \geq 0$ for all $j > i \geq 1$,
and $\xi_0 = \sum_{k = 1}^\infty \xi_k$. In this case,
we have
\[
\frac{2\e^{2\xi_1}}{(1+\e^{2\xi_1})^2}  \sum_{k>j}\xi_k \leq \Var_j (g) \leq \sum_{k>j}\xi_k\,,
\]
 indicating that if $\sum_{j>1}\sum_{k>j}\xi_k < \infty$ to get $\Gamma(p)>0$, the autoregressive process satisfies Assumption \ref{eq:Gammadef}.

\paragraph{Poisson regression for count time series.}
Let $\A=\N$ and $(\xi_j)_{j \geq 0}$ be a sequence of non-positive absolutely summable real numbers, and a constant $c > 0$.
For all $x \in \A^{\Z_-}$, let
\[
v(x) = \exp\bigg(\,\sum_{j=1}^\infty \xi_j \min\{x_{-j}, c\}\bigg). 
\]
For all $a\in \N$ and $x \in \A^{\Z_-}$, the kernel of a Poisson regression model (see \cite{kedem2005regression}) is defined as 
\[
p(a|x) = \frac{\e^{-v(x)}v(x)^a}{a !}.
\]
Applying the mean value theorem to $r\mapsto \e^{-\e^{r}}\e^{ra}/a!$, and maximizing on $r \in (-\infty,0]$ for each $a \in \N$, we 
obtain that $\Gamma(p)$ is finite if $\sum_{j>1}\sum_{k>j}\xi_k < \infty$. 

\paragraph{Hidden Markov chains.}
Let $(X_j)_{j\in\Z}$ be a Markov chain with transition matrix $P$ on a finite alphabet $\A$ and a function $f:\A\rightarrow \B$ where $\B$ is another alphabet with $|\B|<|\A|$. Specifically, this means that certain symbols in 
$\A$ are indistinguishable or merged, which can, for instance, be attributed to a measurement error.
The process $(Y_j)_{j\in\Z}$ defined by $Y_i=f(X_i)$ is a hidden Markov process, and it is usually not Markov of any order. The works of \cite[Theorem 3.1]{Chazottes_Ugalde_2011}, \cite{de2016continuity} and \cite{piraino2020single} for instance, give quantitative information on the variation of $(Y_j)_{j\in\Z}$, allowing us to compute $\Gamma(p)$. In particular, when $P > 0$ and $\A$ is finite, the variation of the projected process decays exponentially fast.
If, in turn, certain symbols of $\B$ become indistinguishable or are merged, the variation of the resulting process will continue to decay exponentially \cite{piraino2020single}.

\paragraph{Convex mixture of Markov chains.}
Let $(\lambda_j)_{j \geq 1}$ be a sequence of non-negative real numbers such that $\sum_{j=1}^\infty \lambda_j = 1$. Define a family of finite-order Markov kernels $p^{\mathsmaller{[k]}}:\A \times \A^{\llbracket -k, -1\rrbracket} \to [0,1], k\ge1$. The kernel for a mixture of Markov chains is defined, for all $a 
\in \mathcal \A$ and $x \in \A^{\Z_-}$, as
\[
p(a|x) = \sum_{j = 1}^\infty \lambda_j\, p^{\mathsmaller{[j]}}\big(a\big|x^{-1}_{-j}\big).
\]
We have $\Gamma(p) \geq \prod_{k \geq 1} \big(1-\sum_{j \geq k} \lambda_j\big)$ which is $>0$ whenever $\sum_{j =1}^\infty j \lambda_j < \infty$ and $\lambda_0 > 0$. This is a general 
class of kernels since all kernels $p$ such that  $\lim_{j\to\infty}\Var_j(p) = 0$
can be represented as a {convex} mixture of Markov chains \cite{kalikow,comets_fernandez_ferrari_2002}.

\paragraph{Gibbs measures.}
Here we assume that $|\A|<\infty$.
Gibbs measures, or Gibbs random fields, originally introduced in the context of statistical physics, form a natural and extensively studied class of examples in probability theory. These measures are particularly relevant to our guessing problem because they are defined via two-sided conditional probabilities (as formalized in \eqref{eq:specification}). Their inherent structure and rich theoretical framework make them well-suited for addressing such probabilistic inference challenges. Here we only consider Gibbs measures on the one-dimensional lattice $\Z$.

Formally, consider a collection $\Phi:=(\Phi_\Lambda)_{\Lambda\Subset \Z}$ of real-valued functions on $\A^\mathbb Z$ such that, for any $x$, $\Phi_\Lambda(x)$ depends only 
on $x_\Lambda$.  This is called an \emph{interaction potential}. We  assume translation invariance, that is, $\Phi_{\theta \Lambda}\circ T=\Phi_\Lambda$, and absolutely summability, that is,
\begin{equation}\label{eq:ass_integrability}
\sum_{\substack{ \Lambda\Subset \Z \\ 0\in\Lambda}}\, \|\Phi_\Lambda\|_{\infty}<\infty.
\end{equation}
The \emph{energy} of a configuration $x$ in the ``box'' $\Lambda$ is $H_\Lambda(x):=\sum_{\Lambda'\Subset\Z:\Lambda'\cap\Lambda\ne\emptyset}\Phi_{\Lambda'}(x)$. 

A Gibbs measure with potential $\Phi$ is given by any measure that satisfies
\begin{equation}\label{eq:specification}
\P(X_\Lambda=x_\Lambda\mid X_{\Lambda^{\!\comp}}=y_{\Lambda^{\!\comp}})\stackrel{\P-\mathrm{a.s.}}{=}\frac{\exp{\left(-\beta H^\Phi_\Lambda(x_\Lambda y_{\Lambda^{\!\comp}})\right)}}{Z^\Phi_{\Lambda,\beta}(y)}\,\,,\quad \Lambda\Subset \Z,
\end{equation}
where $Z^\Phi_\Lambda(y)$ is the normalization factor and $\beta\ge0$ is a constant known as ``inverse temperature''. In other words, they are specified by inside \emph{vs} outside (finite boxes $\Lambda$) conditioning instead of the present \emph{vs} past conditioning of \eqref{eq:conditionaldef}. For finite range potentials, the corresponding Gibbs measure is a $k$-steps Markov chain for some $k$ and with $\Var_0(p)<1$ (see \cite[Theorem 3.5]{georgii2011gibbs}). For a Gibbs measure with an unbounded range potential $\Phi$, we prove the following proposition. 
\begin{prop}\label{prop:suf_gibbs}
Let $\P$ be a Gibbs measure with potential $\Phi$. Define $\Delta(\Phi_\Lambda):=\max\Phi_\Lambda-\min\Phi_\Lambda$. If $\Phi$ is such that
\begin{equation}\label{condition-on-Phi}
\sum_{k\ge1}\,\sum_{i\ge k}\,
\sum_{\substack{ \min\Lambda= 0 \\ \max\Lambda \ge  i}}
\Delta(\Phi_\Lambda)<\infty\,,
\end{equation}
then the corresponding left conditional expectation satisfies Assumption \eqref{eq:Gammadef} and condition \eqref{eq:exponentialUpper}.
\end{prop}

As an example, consider the 1D long-range Ising model, a classical example of a non-Markovian Gibbs measure. Let $\A=\{-1,+1\}$. Fix $\alpha>1$ and consider the potential $\Phi_\Lambda(x)=|i-j|^{-\alpha}x_ix_j$ if $\Lambda=\{i,j\}$ with $i\ne j$, and $\Phi_\Lambda(x)=0$ otherwise. For $\alpha>2$, the process $(X_j)_{j\in \Z}$ is uniquely specified by \eqref{eq:specification} (see for instance \cite[(8.41)]{georgii2011gibbs}). We have that $\sum_{{\min\Lambda= 0, \max\Lambda \ge i}}\Delta(\Phi_\Lambda)=\sum_{\ell\ge i}\Delta(\Phi_{\{0,\ell\}})$ which is $\mathcal{O}(i^{-\alpha+1})$. Thus, as a consequence of Proposition \ref{prop:suf_gibbs}, Assumption \eqref{eq:Gammadef} is satisfied when $\alpha>3$. 

In the framework of Gibbs measures, there exists a natural extension of Markov chains where conditional independence is a two-sided property. This leads to the exploration of concepts analogous to those found in hidden Markov chains. In \cite{Redig2010} it was shown that the resulting process is a Gibbs measure with an interaction potential that decays exponentially. Starting from an exponentially decaying interaction potential, the transformed potential retains this property. 
In the specific case of the one-dimensional long-range Ising model with \(\alpha > 2\), the transformation yields a potential of the same type but with \(\alpha' = \alpha - 1\). Consequently, to ensure that the transformed process satisfies Assumption \eqref{eq:Gammadef}, it is necessary to begin with \(\alpha > 4\).

\section{Proofs}\label{sec:proofs}

\subsection{Proof of Proposition \ref{prop:expupperbound}}

Let $K = |G|+|D|$ and $G \cup D = \{i_1, \ldots i_{\scaleto{K}{4pt}}\}$ with $i_m < i_{\ell}$ if $m < \ell$. For all $d\in A^{G \cup D}$, we have
\begin{align*}
&\P\big(Y_G =d_G, Y_D = d_D \big)-\inf_{c\, \in A^G}\P\big(Y_G =c, Y_D = d_D \big)\\
&\quad \leq \P\big(Y_G =d_G, Y_D = d_D \big)\\
& \quad = \P\big(Y_{i_1} = d_{i_1})\,\prod_{j =2}^K \P\big(Y_{i_j} = d_{i_j}\,|\,Y_{i_m} = d_{i_m}, 1\leq m \leq j-1\big)\\
&\quad \leq \bar{p}^{\,K}.
\end{align*}
Because the last inequality does not depend on $d$, we conclude that
$\beta_{\smallDG} \leq \bar{p}^{|D|+|G|}$ as claimed.

\subsection{Proof of Theorem \ref{theo:probguess}}

From the definition of the risk, we have
\begin{equation*}
R(\hat{f}^n_{\smallDG})  \leq \beta_{\smallDG}\,.
\end{equation*}
Hence, we will consider the case where
$\epsilon < \beta_{\smallDG}$. 
For $\epsilon>0$ and $b\in \A^D$, let
\begin{align*}
& \scaleto{\gamma}{8pt}_{\smallDG}(b, \epsilon) = \\
&\quad \Big\{c \in \A^G : \P(X_G \neq c, X_{D} = b) - \inf_{a\, \in \A^G}\P(X_G \neq a, X_{D} = b)  \leq \epsilon\Big\}.
\end{align*}
We have
\begin{align*}
&\widetilde{\P}\big(\,\hat{f}^n_{\smallDG}(Y_D) \neq  Y_G, Y_D = b\big) -  \inf_{a\, \in A^G}\widetilde{\P}\big(a \neq  Y_G, Y_D = b\big)\\
& = \sup_{a\, \in A^G}\widetilde{\P}\big(Y_G = a, Y_D = b\big) - \widetilde{\P}\big(\hat{f}^n_{\smallDG}(Y_D) =  Y_G, Y_D = b\big) \\
& = \sup_{a\, \in A^G}\widetilde{\P}\big(Y_G = a, Y_D = b\big)\, - \sum_{c\, \in A^G}\widetilde{\P}\big(\hat{f}^n_{\smallDG}(b) =  c, Y_D = b, Y_G = c\big) \\
& = \sum_{c\, \in A^G}\bigg(\,\sup_{a\, \in A^G}\widetilde{\P}\big(Y_G = a, Y_D = b\big) - \widetilde{\P}\big(Y_G = c, Y_D = b\big)\bigg)\,\widetilde{\P}\big(\hat{f}^n_{\smallDG}(b) =  c\,\big) \\
& = \sum_{c\, \in \gamma_{D,G}(b, \epsilon/2)}\bigg(\,\sup_{a\, \in A^G}\widetilde{\P}\big(Y_G = a, Y_D = b\big) - \widetilde{\P}\big(Y_G = c, Y_D = b\big)\bigg)\,\widetilde{\P}\big(\hat{f}^n_{\smallDG}(b) =  c\big) \\
&+\sum_{c\, \notin \gamma_{D,G}(b, \epsilon/2)}\bigg(\sup_{a \in A^G}\widetilde{\P}\big(Y_G = a, Y_D = b\big) - \widetilde{\P}\big(Y_G = c, Y_D = b\big)\bigg)\,\widetilde{\P}\big(\hat{f}^n_{\smallDG}(b) =  c\,\big)\,.
\end{align*}
From the definition of $\scaleto{\gamma}{8pt}_{\smallDG}$, for all $b \in A^D$, we have 
\begin{align*}
& \sum_{c\, \in \gamma_{D,G}(b, \epsilon/2)}\bigg(\,\sup_{a\, \in A^G}\widetilde{\P}\big(Y_G = a, Y_D = b\big) - \widetilde{\P}\big(Y_G = c, Y_D = b\big)\bigg)\,\widetilde{\P}\big(\hat{f}^n_{\smallDG}(b) =  c\big)\\
&\quad \leq \frac{\epsilon}{2}\,\1\{\epsilon \geq 2\delta_{\smallDG}\}\,,
\end{align*}
where in the last inequality, we used the fact that if $\epsilon/2 < \delta_{\smallDG}$ then $\gamma_{\smallDG}(b,\epsilon/2) = \argmin \P(X_G \neq a, X_D = b)$ and $\P\big(\hat{f}^n_{\smallDG}(b) \in  \gamma_{\smallDG}(b, \frac{\epsilon}{2})\big) \leq 1$.

For $\epsilon>0$, let
\[
\beta_{\smallDG}(\epsilon) = \sup_{a\, \in A^G}\widetilde{\P}\big(Y_G = a, Y_D = b\big) - \inf_{c \in A^G}\widetilde{\P}\big(Y_G = c, Y_D = b\big)\,. 
\]
We also have 
\begin{align*}
& \sum_{c\, \notin \gamma_{\smallDG}(b, \epsilon/2)}\bigg(\,\sup_{a\, \in A^G}\widetilde{\P}\big(Y_G = a, Y_D = b\big) - \widetilde{\P}\big(Y_G = c, Y_D = b\big)\bigg)\widetilde{\P}\big(\hat{f}^n_{\smallDG}(b) =  c\big)\\
& \quad \leq \widetilde{\P}\big(\hat{f}^n_{\smallDG}(b) \notin \scaleto{\gamma}{8pt}_{\smallDG}\big(b, 
\scaleto{\frac{\epsilon}{2}}{12pt}\big)\big)\,\beta_{\smallDG}(\epsilon)\,.
\end{align*}

For all $b \in A^D$, we can rewrite
$\scaleto{\gamma}{8pt}_{\smallDG}(b, \epsilon)$ as
\begin{equation*}
\scaleto{\gamma}{8pt}_{\smallDG}(b, \epsilon) = \bigg\{c \in \A^G : \sup_{a\, \in \A^G}\P(X_G = a, X_{D} = b) -\P(X_G = c, X_{D} = b) \leq \epsilon\bigg\}.
\end{equation*}
Hence, if 
\begin{equation}\label{eq:mainbound}
\sup_{a\, \in \A^G\!,\, b\, \in \A^{D}} \Bigg|\, \frac{N^n_{\smallDG}(b,a)}{n} - \P\big(X_G = a, X_{D} = b\big)\,\Bigg|  \leq \frac{\epsilon}{4} \vee \frac{\delta_{\smallDG}}{2}\,,
\end{equation}
then, for all $b \in A^D$,
\begin{equation*}
\hat{f}^n_{\smallDG}(b) \in \scaleto{\gamma}{8pt}_{\smallDG}\big(b, \scaleto{\frac{\epsilon}{2}}{12pt}\big)\,.
\end{equation*}
It remains to bound the probability from above that \eqref{eq:mainbound} is not satisfied. For this, we use  Theorem \ref{theo:DKW} which gives us the following  Dvoretzky-Kiefer-Wolfowitz type inequality:
for any $u>0$ and $n\ge L=L_{\smallDG}=\mathrm{diam}(G\cup D)$,
\begin{align}\label{eq:CGT2022}
&\P\left(\,\sup_{a\, \in \A^G, \, b\, \in \A^{D}}\bigg|\frac{N^n_{\smallDG}(b,a)}{n} - \P\big(X_G = a, X_{D} = b\big)\bigg|  > u + \sqrt{\frac{2M(1-\Gamma)+\Gamma}{\Gamma(n-L+2)}} \,\right) \notag \\
&\leq \exp \Big(-2(n-L+{2})\,\Gamma^2M^{-2}u^2\Big),
\end{align}
where we used $M:=|D|+|G|$, and wrote
$\Gamma$ for $\Gamma(p)$ to alleviate notation.
 Equating the upper bound of \eqref{eq:CGT2022} to $\epsilon/(2\beta_{\smallDG}(\epsilon))$ 
 (when $\epsilon < 2\delta_{\smallDG}$ we could equate to $\epsilon/\beta_{\smallDG}(\epsilon)$, but the gain is insubstantial), we obtain
\begin{equation}\label{value-of-u}
u = \frac{M}{\Gamma}\sqrt{\frac{\log\left( \frac{2\beta_{\smallDG}(\epsilon)}{\epsilon}\right)}{2(n-L+{2})}}\;.
\end{equation}

If $n$ satisfies \eqref{eq:samplesize}, we obtain
\[
\frac{(\epsilon/2 \vee \delta_{\smallDG})^2}{4}\geq 2\Bigg(u^2+\frac{2M(1-\Gamma)+\Gamma}{\Gamma( n-L+2)}\Bigg).
\]
Now, using the elementary inequality $(\sqrt{u}+\sqrt{v}\,)^2\leq 2(u+v)$ for $u,v\ge0$, we conclude that
\[
\frac{\epsilon}{4} \vee \frac{\delta_{\smallDG}}{2}\geq u+\sqrt{\frac{2M(1-\Gamma)+\Gamma}{\Gamma( n-L+2)}}\,,
\]
which finishes the proof.


\subsection{Proof of Corollary \ref{coro:rate}}

Take $\epsilon_n = \frac{1}{2}\sqrt{\frac{n}{\log n}}$. If \eqref{eq:subcritical} holds, there exists $m_0 \geq 2$ such that for all $n \geq m_0$, $2\delta_n < \epsilon_n$ and thus, in \eqref{eq:samplesize}, $\frac{\epsilon_n}{2} \vee \delta_{\smallDG}=\frac{\epsilon_n}{2}$. Moreover,  there is $m_1 \geq 2$ such that for all $n \geq m_1$ the right-hand side of  \eqref{eq:samplesize} is smaller than $n$ with $\epsilon=\epsilon_n$. Therefore, taking $n_0 = m_0 \vee m_1$, and using Theorem \ref{theo:probguess}, we conclude the proof of the first statement of the corollary.

Similarly, take $\epsilon_n=\frac{1}{2}\sqrt{\frac{n}{\log n}}$. If \eqref{eq:supercritical}  holds, there exists $n_0 \geq 2$ such that for all $n \geq n_0$, $2\delta_n \geq \epsilon_n$ and \eqref{eq:samplesize} is satisfied. So applying Theorem \ref{theo:probguess} concludes the proof of the second statement of the corollary. 

\subsection{Proof of Theorem \ref{theo:lowerbound}}

Our proof follows a classical procedure to obtain minimax lower bounds for estimators. Let us first write the minimax problem. 
Let $n\geq 1$. Writing $\widetilde{\P} = \P \otimes \P$, we have, by definition, for any set of probability measures $\EuScript{M}$
\begin{align*} 
&\inf_{\psi_n \in \Psi_n}\sup_{\P \in \EuScript{M}}R(\psi_n;\P) \\
&= \inf_{\psi_n \in \Psi_n}\sup_{\P \in \EuScript{M}}\; \sup_{b\, \in \A^D}\Big(\widetilde{\P}\big(\psi_n(X_1^n, b) \neq  Y_G, Y_D = b\big) - \inf_{a\, \in \A^G}\widetilde{\P}\big(a \neq  Y_G, Y_D = b\big)\Big). \notag
\end{align*} 
Now, we want to obtain a lower bound for the left-hand side of the above equation. Fix an element $b \in \A^D$. We will later choose a specific $b$. Let $a_{\sP}= \argmin_{a \in \A^G}\widetilde{\P}(a \neq  Y_G, Y_D = b)$. A minimizer exists because the probability sums to one. If there is more than one minimizer, we can just choose one of them. 
To simplify the notation, let us abbreviate \(\psi_n(X_1^n, b)\) as \(\psi_n(b)\)
in the next estimations. We have
\begin{align*} 
&\widetilde{\P}\big(\psi_n(b) \neq  Y_G, Y_D = b\big) - \widetilde{\P}\big(a_{\sP} \neq  Y_G, Y_D = b\big)\\
&= \sum_{c\, \neq \, a_{\P}} \widetilde{\P}\big(\psi_n \neq  c,  Y_G = c, Y_D = b\big) - \sum_{c\, \neq\, a_{\P}} \widetilde{\P}\big(Y_G = c, Y_D = b\big)\\ 
&\hspace{1cm}+ \widetilde{\P}\big(\psi_n(b) \neq  a_{\sP}, Y_G = a_{\sP}, Y_D = b\big)\\
&= \widetilde{\P}\big(\psi_n(b) \neq  a_{\sP}, Y_G = a_{\sP}, Y_D = b\big) -
\sum_{c\, \neq\, a_\P} \widetilde{\P}\big(\psi_n(b) =  c,  Y_G = c, Y_D = b\big)\\
&=  \sum_{c\, \neq\, a_\P} \widetilde{\P}\big(\psi_n(b) =  c, Y_G = a_{\sP}, Y_D = b\big) - \sum_{c\, \neq\, a_{\sP}} \widetilde{\P}\big(\psi_n(b) =  c,  Y_G = c, Y_D = b\big)\\
&=  \sum_{c\, \neq \, a_\P} {\P}\big(\psi_n(b) =  c\big)\Big(\P\big(Y_G = a_{\sP}, Y_D = b\big) - {\P}\big(Y_G = c, Y_D = b\big)\Big).
\end{align*} 
Let define the margin $\delta_{\sP}(b) : = \inf\{\P(Y_G = a_{\sP}, Y_D = b) - {\P}(Y_G = c, Y_D = b) > 0: c \neq a_{\sP}\}$. If the set inside the infimum is empty, we take $\delta_\P(b) = 0$. Let $\P_0$ and $\P_1$ be any two probability distributions in $\EuScript{M}$. 
We have, for all $b \in A^D$, 
\begin{align} 
& \inf_{\psi_n \in\, \Psi\!_n}\sup_{\P \in \EuScript{M}}R(\psi_n) \\
& \geq \inf_{\psi_n \in\, \Psi\!_n}\sup_{\P \in \EuScript{M}} \delta_{\sP}(b)\,  {\P}\big(\psi_n(X_1^n, b) \neq  a_{\sP}\big) \notag\\
& \geq \inf_{\psi_n \in\, \Psi\!_n}\sup_{i\, \in \{0,1\}} \delta_{\sP_i}(b)  {\P}_i\big(\psi_n(X_1^n, b) \neq  a_{\sP_i}\big) \notag\\
& \geq \left(\inf_{i\, \in \{0,1\}}\delta_{\sP_i}(b)\right) \inf_{\psi_n \in\, \Psi\!_n}
\,\sup_{i\, \in \{0,1\}}  {\P}_i\big(\psi_n(X_1^n, b) \neq  a_{\sP_i}\big)\notag \\
& \geq \left(\inf_{i\, \in \{0,1\}}\delta_{\sP_i}(b)\right) \inf_{\psi_n \in \,\Psi\!_n}
\frac{{\P}_0\big(\psi_n(X_1^n, b) \neq  a_{\sP_0}\big)+{\P}_1\big(\psi_n(X_1^n, b) \neq  a_{\sP_1}\big)}{2}. \label{eq:reduction}
\end{align} 
For $x \in \A^n$, let $\phi_*(x,b) = a_{\sP_0}$ if $\P_0(X_1^n = x) \geq \P_1(X_1^n = x)$ and  $\phi_*(x,b) = a_{\sP_1}$ otherwise.
It is straightforward to verify that the function $\phi_*$  minimizes the righthand side of the inequality \eqref{eq:reduction}. Let $\P_0^{(n)}$ and $\P_1^{(n)}$ be the projection of $\P_0$ and $\P_1$ on $\{1, \ldots, n\}$. From the definition of total variation distance $d_{{\scriptscriptstyle \mathrm{TV}}}$, we have that
\begin{equation*}
\frac{{\P}_0(\phi_*(X_1^n, b) \neq  a_{\sP_0})+{\P}_1(\phi_*(X_1^n, b) \neq  a_{\sP_1})}{2} = 1-d_{{\scriptscriptstyle \mathrm{TV}}}\big(\P_0^{(n)}, \P_1^{(n)}\big).
\end{equation*}

To get a lower bound of the right-hand side, we use Bretagnolle-Huber inequality \cite[Lemma 2.6, p. 89] {tsybakov2009nonparametric}, that is,
\begin{equation*} 
1-d_{{\scriptscriptstyle \mathrm{TV}}}\big(\P_0^{(n)}, \P_1^{(n)}\big) \geq \e^{-\KL\big(\P_0^{(n)} \Vert\P_1^{(n)}\big)},
\end{equation*} 
where $\KL\big(\P_0^{(n)} \Vert\P_1^{(n)}\big)$ is the Kullback-Leibler divergence, that is,
\begin{align*} 
& \KL\big(\P_0^{(n)} \Vert \P_1^{(n)}\big): = \sum_{x^n_1 \in \A^n} \P_0^{(n)}(X_1^n = x^n_1)\log \frac{\P_0^{(n)}(X_1^n = x^n_1)}{\P_1^{(n)}(X_1^n = x^n_1)}\\
&= \sum_{x^n_1 \in \A^n} \P_0^{(n)}(X_1^n = x^n_1) \sum_{k=0}^{n-1}\log\frac{\P_0^{(n)}\big(X_{n-k}=x_{n-k}\,|\,X_1^{n-k-1} = x_1^{n-k-1}\big)}{\P_1^{(n)}\big(X_{n-k}=x_{n-k}\,|\,X_1^{n-k-1} = x_1^{n-k-1}\big)}\\
&=\sum_{k= 0}^{n-1}\E_{\P_0^{(n)}}\left[\sum_{a \in \A} \P_0^{(n)}\big(X_{n-k} = a 
\,\big|\,X_1^{n-k-1}\big)
\log\frac{\P_0^{(n)}\big(X_{n-k}=a\,|\,X_1^{n-k-1}\big)}{\P_1^{(n)}\big(X_{n-k}=a\,\big|\,X_1^{n-k-1}\big)}\right].
\end{align*} 
Using Lemma \ref{lemma:KLchi} (see Appendix \ref{app:KL}), we obtain 
\begin{align*} 
\KL\big(\P_0^{(n)} \Vert \P_1^{(n)}\big)\leq \sum_{k=0}^{n-1}\; \sup_{x,y,z \in \A^{\Z_-}}
\mathlarger{\sum}_{a \in \mathcal{A}} \frac{\Big(p_{\sP_0}\big(a\,|\,x^{-1}_{-k}y\big)-p_{\sP_1}\big(a\,|\,x^{-1}_{-k}z\big)\Big)^2}{p_{\sP_1}(a\,|\,x^{-1}_{-k}z)}.
\end{align*} 

Putting everything together, we obtain
\begin{align*}
& \inf_{\psi_n \in \Psi_n}\sup_{\P \in \EuScript{M}}R(\psi_n;\P)\\
& \geq  \left(\inf_{i \in \{0,1\}}\delta_{\sP_i}(b)\right) \exp\left(-\sum_{k=0}^{n-1} \;\sup_{x,y,z \in \A^{\Z_-}}
 \mathlarger{\sum}_{a \in\mathcal{A}}\frac{(p_{\sP_0}\big(a\,|\,x^{-1}_{-k}y\big)-p_{\sP_1}\big(a\,|\,x^{-1}_{-k}z\big)\big)^2}{p_{\sP_1}\big(a\,|\,x^{-1}_{-k}z\big)}\right).
\end{align*}

Let index the elements of the alphabet as $\mathcal{A} = \{a_0, a_1, a_2, \ldots, a_{|\A|}\}$ when the alphabet is finite, otherwise write $\mathcal{A} = \{a_0, a_1, a_2, \ldots \}$. Choose two distinct elements $a_0,a_1\in\A$. 

We will first consider the case when $\EuScript{M} = \EuScript{P}_n$. Let $\P_0, \P_1 \in \EuScript{P}_n$ be product measures with $\P_0(X_1 = a_{0}) = 4^{-1} + 1/(8\sqrt{n}\,)+2^{-|\mathcal{A}|}$, $\P_0(X_1 = a_{1}) = 4^{-1} - 1/(8\sqrt{n}\,) +2^{-|\mathcal{A}|}$ and, for $i \geq 2$, $\P_0(X_1 = a_i) = 2^{-i}$. Similarly $\P_1(X_1 = a_{1}) = 4^{-1} + 1/(8\sqrt{n}\,) +2^{-|\mathcal{A}|}$, $\P_1(X_1 = a_{0}) = 4^{-1} - 1/(8\sqrt{n}\,) +2^{-|\mathcal{A}|}$ and for $i \geq 2 $, $\P_1(X_1 = a_i) = 2^{-i}$. When $|\A| = \infty$, we have $2^{-|\A|} = 0$. 

With these choices, the supremum in the argument of the exponential simplifies to
\begin{align*}
\sum_{a\, \in \mathcal{A}} \frac{\big(\,p_{\sP_0}(a)-p_{\sP_1}(a)\big)^2}{p_{\sP_1}(a)} \leq \frac{1}{n}\,.
\end{align*}
Also
\begin{align*}
&\delta_{\sP_0}(b)=\delta_{\sP_1}(b)=\\
& \inf_{\substack{c\,\in \A \\ c\neq a_0}}\big[\P\!_0\big(Y_G = (a_0, a_0, \ldots, a_0)\big) - {\P}\!_0\big(Y_G = (c, a_0\ldots,a_0)\big)\big]\,\P_1( Y_D = b).
\end{align*}
As we are free to choose any $b$, we take $b = (a_0, \ldots, a_0) \in A^D$. We obtain
\[
\delta_{\sP_0}\ge \frac{1}{4\sqrt{n}} \left(\frac{1}{4}+\frac{1}{8\sqrt{n}}\right)^{|D|+|G|-1}\geq \frac{1}{\sqrt{n}}\left(\frac{1}{4}\right)^{|D|+|G|},
\]
and we conclude that 
\begin{equation*} 
\inf_{\psi_n \in \Psi_n}\sup_{\P \in \EuScript{P}\!_n}R(\psi_n;\P) \geq \frac{1}{\sqrt{n}}\left(\frac{1}{4}\right)^{|D|+|G|}\e^{-1}\,.
\end{equation*} 

 Now, we consider the case where $\EuScript{M} = \EuScript{Q}_n$.
 Let $\P_0, \P_1 \in \EuScript{Q}_n$ be the product measures with $\P_0(X_1 = a_{0}) = 4^{-1} + \delta_n/8+2^{-|\mathcal{A}|}$, $\P_0(X_1 = a_{1}) = 4^{-1} - \delta_n/8 +2^{-|\mathcal{A}|}$ and, for $i \geq 2$, $\P_0(X_1 = a_i) = 2^{-i}$. Similarly $\P_1(X_1 = a_{1}) = 4^{-1} + \delta_n/8 +2^{-|\mathcal{A}|}$, $\P_1(X_1 = a_{0}) = 4^{-1} - \delta_n/8 +2^{-|\mathcal{A}|}$ and for $i \geq 2 $, $\P_1(X_1 = a_i) = 2^{-i}$. 

With these choices, the supremum in the argument of the exponential can be bounded from above by
\begin{align*}
\sum_{a\, \in \mathcal{A}} \frac{\big(\,p_{\sP_0}(a)-p_{\sP_1}(a)\big)^2}{p_{\sP_1}(a)} \leq \delta_n^2\,.
\end{align*}
Because
\begin{align*}
&\delta_{\sP_0}=\delta_{\sP_1}=\\
& \inf_{\substack{c\,\in \A \\ c\neq a_0}}\big[\P_0\big(Y_G = (a_0, a_0, \ldots, a_0)\big) - {\P}_0\big(Y_G = (c, a_0\ldots,a_0)\big)\big]\,\P_1( Y_D = b),
\end{align*}
choosing $b = (a_0, \ldots, a_0) \in A^D$, we obtain
\[
\delta_{\sP_0}\ge \frac{\delta_n}{4} \left(\frac{1}{4}+\frac{\delta_n}{8}\right)^{|D|+|G|-1}\geq \delta_n\left(\frac{1}{4}\right)^{|D|+|G|}.
\]
Therefore,
\begin{equation*} 
\inf_{\psi_n \in \Psi_n}\sup_{\P \in \EuScript{Q}_n}R(\psi_n;\P) \geq \delta_n\left(\frac{1}{4}\right)^{|D|+|G|}\e^{-n\delta_n^2}\,.
\end{equation*} 

\appendix

\section{Dvoretzky-Kiefer-Wolfowitz type inequality}

The following theorem is a restatement of the results of \cite[Theorem 4.1]{chazottes2023gaussian} to better fit the needs of this paper. The proof is included for the convenience of the reader. This result is a variant of what is commonly referred to in the literature as the Dvoretzky-Kiefer-Wolfowitz inequality, specifically for the case of independent random variables.

Let $S \Subset \Z$ and $\sigma\in\A^S$ with $\text{diam}(S):=\sup S - \inf S = k \,(< \infty)$.
For any $n \geq k$, let us denote the frequency of occurrences of the string
$\sigma$ in $X_1^n$ by
\[
\hat r_S^n(\sigma):=\frac{N^n_{S}(\sigma)}{n-k+2}: = 
\frac{\sum_{i = 0}^{n-k+1} \1 \{X_{S+i}=\sigma\}}{n-k+2}.
\]

\begin{thm} \label{theo:DKW}
Let $(X_j)_{j \in \Z}$ be a stationary process such that \eqref{eq:Gammadef} holds. Then, for all $u > 0$, for all $n > 0$, and $0< k\leq n$, we have
\begin{align*}
&\mu\left(\;\sup_{a\,\in\A^S}\big|\,\hat r_S^n(a)-\P(X_S=a)\big| > u+ \sqrt{\frac{2|S|(1-\Gamma (p))+\Gamma (p)}{\Gamma (p)(n-k+2)}}\,\right)\\ &\leq \exp\left(-\frac{2(n-k+2)\, u^2}{|S|^{2}\, \Gamma (p)^{-2}} \right).
\end{align*}
\end{thm}
\begin{proof} 
Define the statistic ${f}_S= \|\,\hat r_S^n(\cdot) -\P(X_S\in\cdot)\|_\infty$. Without loss of generality, we assume $\inf S = 1$.
By the Gaussian concentration bound in \cite[Theorem 3.2]{chazottes2023gaussian}, we have
\begin{equation}\label{vendredi}
\P\big(f_S -\E[f_S] > u\big)\leq \exp \big( -2 |S|^{-2}(n-k+2)\, \Gamma (p)^{2}u^2\,\big).
\end{equation}
Therefore, to prove Theorem  \ref{theo:DKW}, it remains to find a good upper bound for $\E[f_S]$.
Here, we follow the argument used in \cite{kontorovich2014uniform}.
By Jensen's inequality, and since $\E[\,\hat r_S^n(\sigma)]=\P(X_S=\sigma)$, we have
\begin{align}
\big(\E[{f}_S]\big)^2
&\leq \E\big[{f}_S^2\big] \notag 
\leq \E\left[\,\sum_{\sigma \in \A^S}\big(\,\hat r_S^n(\sigma) -\P(X_S=\sigma)\big)^2\right] \notag \\
&\leq \mathlarger{\sum}_{\sigma \in \A^S}  \Big(\E\big[\hat r_S^n(\sigma)^2\big] -\P(X_S=\sigma)^2\Big).
\label{eq:DKW1}
\end{align}
For all $\sigma \in \A^S$, we have 
\begin{align*}
&\E\big[\,\hat r_S^n(\sigma)^2\big]\\
&= \frac{1}{(n-k+2)^2}\,\E\!\left[\left(\,\sum_{i=0}^{n-k+1} \1 \{X_{S+i}=\sigma\}\right)^2\,\right]\\
& =\frac{1}{(n-k+2)^2}\,\E\!\left[\,\sum_{i=0}^{n-k+1}  \1 \{X_{S+i}=\sigma\}\right]\\
& \qquad \qquad \qquad \quad \quad \;+ \frac{2}{(n-k+2)^2}\,\E\!\left[\,\sum_{j=1}^{n-k+1}\sum_{i=0}^{j-1}\,  \1 \{X_{S+i}=\sigma\}\, \1 \{X_{S+j}=\sigma\}\right].
\end{align*}
Hence, by stationarity, we get
\begin{align*}
&\E\big[\,\hat r_S^n(\sigma)^2\big]\\
& = \frac{\P(X_S = \sigma)}{n-k+2} + \frac{2}{(n-k+2)^2}\sum_{j=1}^{n-k+1}\sum_{i=0}^{j-1}\,\P\big(X_{S+i} = \sigma, X_{S+j} = \sigma\big)\\
& = \frac{\P(X_S = \sigma)}{n-k+2} +
\frac{2}{(n-k+2)^2}\sum_{j=1}^{n-k+1}\sum_{i=0}^{j-1}\,\P(X_S =\sigma)\,\P\big(X_{S+j} = \sigma \,\big|\, X_{S+i} = \sigma\big)\\
& \leq \frac{\P(X_S = \sigma)}{n-k+2} + \frac{2}{(n-k+2)^2}\sum_{j=1}^{n-k+1}\sum_{i=0}^{j-1}
\Big[\P(X_S = \sigma)\times \\
& \qquad \qquad\qquad \quad\;\; \big(\P(X_S = \sigma) +\;\big|\P\big(X_{S+j} = \sigma \,\big|\, X_{S+i} = \sigma\big) - \P(X_S = \sigma)\,\big|\,\big)\Big].
\end{align*}
Now, for all $\sigma \in \A^S$, we have
\begin{align*}
& \big|\,\P\big(X_{S+j} = \sigma \mid X_{S+i} = \sigma\big) - \P(X_S = \sigma)\big| \\
& \quad \leq \sup_{\tilde{\sigma} \in \A^S}|\,\P(X_{S+j} = \sigma \mid X_{S+i} = \sigma) - \P(X_{S+j} = \sigma \mid X_{S+i} = \tilde{\sigma})|\\
& \quad = \sup_{\tilde{\sigma} \in \A^S}
\mid\P\big(X_{S+j-i} = \sigma \mid X_{S} = \sigma\big) -
\P\big(X_{S+j-i} = \sigma \mid X_{S} = \tilde{\sigma}\big)\mid\\
& \quad \leq \sup_{x,y\, \in \A^{\Z_-}}
\mid\P^{x}\big(X_{S+j-i} = \sigma \big) - \P^{y}\big(X_{S+j-i} = \sigma \big)\big|\\
& \quad \leq \sup_{x,y\, \in \A^{\Z_-}} \sum_{\ell\, \in S+j-i} \P^{x,y}\big(\eta_\ell \neq \omega_\ell\big)=:q_{j-i}(S)\,,
\end{align*}
where $\P^{x,y}$ is the one-step maximal coupling between $\P^x$ and $\P^y$, which are the laws of the process $(X_i)_{i\in\Z}$ conditioned on starting with pasts $x$ and $y$, respectively. 
Coming back to the estimation of $\E\left[\hat r_S^n(\sigma)^2\right]$, we have
\begin{align}
\nonumber
& \E\!\left[\hat r_S^n(\sigma)^2\right]- \frac{\P(X_S = \sigma)}{n-k+2}\\
\nonumber
& \leq \frac{2}{(n-k+2)^2}\sum_{j=1}^{n-k+1}\,\sum_{i=0}^{j-1}\P(X_S = \sigma)\big(\P(X_S = \sigma) +
q_{j-i}(S)\big) \\
\nonumber
& \leq  \P(X_S = \sigma)^2+ \frac{2\,\P(X_S = \sigma)}{(n-k+2)^2}\sum_{j=1}^{n-k+1}\,\sum_{i=0}^{j-1} 
q_{j-i}(S) \\
\nonumber
& \leq  \P(X_S = \sigma)^2+ \frac{2\,\P(X_S = \sigma)}{(n-k+2)^2}\sum_{j=1}^{n-k+1}\,\sum_{i\ge1} 
q_{i}(S) \\
\nonumber
& \leq  \P(X_S = \sigma)^2+ \frac{2\,\P(X_S = \sigma)}{n-k+2}\,\sum_{i\ge1} 
q_{i}(S)\,. 
\nonumber
\end{align}
On the other hand, 
\[
\sum_{i\,\ge1} q_{i}(S)=\sum_{i\,\ge1}\;\sup_{x,y\, \in \A^{\Z_-}} \sum_{\ell\, \in S+i} \P^{x,y}\big(\eta_\ell \neq \omega_\ell\big)\le|S|\sum_{i\,\ge1} \;\sup_{x,y\, \in \A^{\Z_-}}\P^{x,y}\big(\eta_\ell \neq \omega_\ell\big)\,.
\]
So we only need an upper bound for the coupling errors. Proposition 5.2 in  \cite{chazottes2023gaussian} gives us that 
\begin{equation*}
\sum_{i=1}^{\infty}\; \sup_{x,y\, \in\, \A^{\Z_-}}\!\!\P^{x,y}
(\eta_{i} \neq \omega_{i}) \leq \frac{1-\Gamma(p)}{\Gamma(p)}.
\end{equation*}
Hence, we have that
\begin{equation}
\label{coucou}
 \E\!\left[\,\hat r_S^n(\sigma)^2\right] -\P(X_S = \sigma)^2\leq  \frac{2\, |S|(1-\Gamma (p))\,\Gamma (p)^{-1} +1}{n-k+2} \,\P(X_S = \sigma).
\end{equation}
Using \eqref{eq:DKW1} we obtain 
\[
\E\big[\|\,\hat r_S^n -\P\|_\infty\big] \leq \sqrt{\frac{2|S|(1-\Gamma (p))\,\Gamma (p)^{-1} + 1}{n-k+2}}\,,
\]
which is the desired bound.
\end{proof}

\section{One-dimensional Gibbs measures}

The goal of this section is to establish Theorem \ref{prop:gibbsvsg} below, from which Proposition \ref{prop:suf_gibbs} follows as a direct corollary. To this end, we return to the framework introduced in the example of Gibbs measures in Section \ref{sec:examples}. Let \((X_j)_{j \in \mathbb{Z}}\) be a Gibbs measure associated with the potential \(\Phi\), and denote by \(p_\Phi\) the regular version of its left conditional expectations, as defined in \eqref{eq:conditionaldef}. 

Observe that by \eqref{eq:ass_integrability}, the conditional probabilities \eqref{eq:specification} are uniformly bounded away from zero, that is 
\begin{equation}
\label{eq:non_nullness}
h_\Lambda:=\inf_{x,y}\P(X_\Lambda=x_\Lambda|X_{\Lambda^{\!\comp}}=y_{\Lambda^{\!\comp}})>0\,,
\end{equation} 
uniformly in $\Lambda$, and in particular, $h_{\{0\}}>0$.

The next theorem \ref{prop:gibbsvsg} establishes a relationship between the variation of \(p_\Phi\) and \(\Delta(\Phi_\Lambda) := \max \Phi_\Lambda - \min \Phi_\Lambda\) (oscillation of $\Phi_\Lambda$). 

\begin{thm}\label{prop:gibbsvsg}
We have $\Var_0(p_\Phi)\le 1- h_{\{0\}}$ and 
\begin{equation}\label{eq:boundingVar}
\Var_k(p_\Phi)\le \frac{|\A|}{2}\sum_{i\ge k}\sum_{\substack{\min\Lambda= 0\\ \max\Lambda \ge \frac i2}}\Delta(\Phi_{\Lambda}),\, k\geq 1.
\end{equation}
\end{thm}

Proposition \ref{prop:suf_gibbs} follows directly as a corollary of this result, as \(\Var_k(p_\Phi)\) is a non-increasing sequence. Moreover, for a \((0,1)\)-valued sequence of real numbers \((a_n)_{n \in \mathbb{N}}\), it holds that \(\sum_n a_n < \infty \iff \prod_n (1 - a_n) > 0\). 
Furthermore, defining
\[
n_\star:=\inf\Bigg\{n\ge0:\frac{|\A|}{2}\sum_{i\,\ge\, n}\sum_{\substack{ \min\Lambda= 0\\ \max\Lambda \ge \frac i2}}\Delta(\Phi_\Lambda)<1\Bigg\}\,,
\]
we also see that
\begin{equation}
\Gamma(p_{\Phi})\ge\mathlarger{\prod}_{n\ge n_\star}\Bigg(1-\frac{|\A|}{2}\sum_{i\,\ge\, n}\sum_{\substack{ \min\Lambda= 0\\ \max\Lambda \ge \frac i2}}\Delta(\Phi_\Lambda)\Bigg),
\end{equation}
which, in principle, makes it possible to derive explicit bounds on \(\Gamma(p_{\Phi})\).

Before delving into the proof of this theorem, let us offer a brief prelude.
For a given function $\phi$, defined either on $\A^\Z$ or $
\A^{\Z_-}\times\A$, an \emph{equilibrium state} for $\phi$ is a measure $\mu_\phi$ on $\A^\Z$, shift-invariant (that is $\mu_\phi\circ T^{-1}=\mu_\phi$ where $T$ is the left shift) which maximizes $\int \phi \dd\mu+h_\mu(T)$ where $h_\mu(T)$ denotes the entropy of $\mu$. The function $\phi$ is called two-sided (\emph{resp.} one-sided) potential function if it is defined on $\A^\Z$ (\emph{resp.} on $\A^{\Z_-}\times \A$). 
 
The regularity of a potential is measured by its ``variation'' (modulus of continuity with respect to the usual distance).
For $n\geq 0$, let
\[
\text{var}_n(\phi):=\sup\big\{|\phi(x)-\phi(y)|:x,y\in \A^\Z, x_i=y_i,|i|\le n\big\},
\]
for a two-sided potential $\phi$, and 
\[
\text{var}_n(\varphi):=\sup\big\{|\varphi(xa)-\varphi(ya)|:x,y\in \A^{\Z_-},a\in \A, x_i=y_i,i\ge-n\big\}\,,
\]
for a one-sided potential $\varphi$. 

Theorem \ref{prop:gibbsvsg} establishes a connection between the variation \(\Delta(\Phi)\) of the potential in a Gibbs (two-sided) specification and the variation \(\Var_n(p)\) of the one-sided conditional expectations. In essence, it links two distinct methods of specifying stationary processes: the two-sided conditioning used in Gibbs measures and the one-sided conditioning employed in stochastic processes. The proof strategy leverages equilibrium states from dynamical systems, offering a third perspective on defining stationary processes and serving as a bridge between the two conditioning frameworks.

\begin{proof}[Proof of Theorem \ref{prop:gibbsvsg}] Denote by $\mu$ the law of the process $(X_j)_{j \in \Z}$, that is, the measure which is defined through $\mu(C)=\P(X_{-\infty}^{+\infty}\in C)$ for any measurable set $C$. 
Since $(X_j)_{j \in \Z}$ is stationary, $\mu$ is invariant under the shift ($\mu(C)=\mu(T(C))$ for any measurable set $C$).
We know by assumption that $(X_j)_{j \in \Z}$ (or equivalently $\mu$) satisfies \eqref{eq:specification}.
 It is proven in \cite[Section 3]{coelho_quas_1998} that a two-sided potential $\phi$ can be constructed from $\Phi$, which is such that $\mu$ is an equilibrium state for $\phi$ and moreover 
\begin{equation}\label{eq:step1}
\text{var}_n(\phi)\le \sum_{\substack{\min\Lambda= 0\\ \max\Lambda \ge n}}\Delta(\Phi_\Lambda). 
\end{equation}
We next use \cite[Section 5]{pollicott2000rates} which allows us to construct a one-sided potential $\varphi$ out of the two-sided potential $\phi$, in such a way that $\mu$ is an equilibrium state for $\varphi$ and moreover
\begin{equation}\label{eq:step2}
\text{var}_n(\varphi)\le\text{var}_{\frac{n}{2}}(\phi). 
\end{equation}
We now invoke \cite[Section IV]{ledrappier/1974} (see also \cite[Proof of Theorem 3.3]{walters/1975}), which  constructs a one-sided normalized  function $g$ (normalized means that it satisfies
$\sum_{a\in \A}g\big(x_{-\infty}^{-1}a\big)=1$, for any $x_{-\infty}^{-1}$) such that $\inf g>0$, and $\mu$ is an equilibrium state for $\log g$. Moreover
\begin{equation}\label{eq:step3}
\text{var}_n(\log g)\le \sum_{k\ge n}\text{var}_k(\varphi).
\end{equation}
Now, by the mean value theorem applied to the exponential function, we have for all $n$,
$x,y\in \A^{\Z_-},a\in \A$, such that $x_i=y_i,i\ge-n$, 
\[
|\, g(xa)-g(ya)|=\big|\e^{\log g(xa)}-\e^{\log g(ya)}\big|\leq
|\log g(xa)-\log g(ya)|,
\]
where we used that $0<g<1$. Thus
\begin{equation}\label{eq:step4}
\text{var}_n(g)\le \text{var}_n(\log g),\, n\geq 1\,.
\end{equation}

To conclude, \cite{ledrappier/1974} also tells us that  $\mu$ being an equilibrium state for $\log g$ is equivalent to saying that the stationary process of the corresponding equilibrium state satisfies \eqref{eq:conditionaldef} with $g(xa)$ in place of $p(a|x)$. But by the a.s. uniqueness of the regular version, we conclude that $g(xa)=p_{\Phi}(a|x)$ $\P$-a.s., and thus $\text{var}_n(g)=\text{var}_n(p_\Phi)$. 

Combining \eqref{eq:step1}, \eqref{eq:step2}, \eqref{eq:step3}, and \eqref{eq:step4}, and noting that
\(
\Var_n(p) \leq \frac{|\A|}{2} \text{var}_n(p)
\), we thus proved \eqref{eq:boundingVar}.

It only remains to prove that $\Var_0(p_\Phi)\le 1- h_{\{0\}}$. We claim that $1-\text{Var}_0(p_\Phi)=\inf_{x,y\in \A^{\Z_-}}\sum_{a\in\A} p_\Phi(a|x)\wedge p_\Phi(a|y)\ge \inf_{a\in\A,x\in \A^{\Z_-}}p_\Phi(a|x)\ge h_{\{0\}}$.
Indeed, for any $x\in \A^{\Z_-}$,
\begin{align*}
\inf_{a\in\A,\,x\in \A^{\Z_-}}p_\Phi(a|x)
& =\inf_{a\in\A,\,x\,\in \A^{\Z_-}} \mathlarger{\int}_{\A^\N}\frac{\exp{\left(-H^\Phi_{\{0\}}(xay)\right)}}{Z^\Phi_\Lambda((xay)_{\Z\setminus\{0\}})}\dd\P(y)\\
& \ge \mathlarger{\int}_{\A^\N}\;\inf_{a\in\A,\,x\in \A^{\Z_-}}\frac{\exp{\left(-H^\Phi_{\{0\}}(xay)\right)}}{Z^\Phi_\Lambda((xay)_{\Z\setminus\{0\}})}\dd\P(y)\\
& \ge \mathlarger{\int}_{\A^\N}\;\inf_{xay\,\in \A^{\Z}}\frac{\exp{\left(-H^\Phi_{\{0\}}(xay)\right)}}{Z^\Phi_\Lambda((xay)_{\Z\setminus\{0\}})}\dd\P(y)\\
& \ge \inf_{xay\,\in \A^\Z} \frac{\exp{\left(-H^\Phi_{\{0\}}(xay)\right)}}{Z^\Phi_\Lambda(y)}
=h_{\{0\}}.
\end{align*}
To prove that condition \eqref{eq:exponentialUpper} is satisfied, we have that $\sup_{a\in\A,\,x\in \A^{\Z_-}}p_\Phi(a|x)\leq 1-h_{\{0\}} < 1$. The theorem is proved.
\end{proof}

\section{A lemma}\label{app:KL}

The following lemma is a classical upper bound for the Kullback-Leibler divergence using chi-square distance.
\begin{lem} \label{lemma:KLchi}
Let $\P$ and $\Q$ be two probability measures on a countable alphabet $\A$. 
Then
\begin{equation*}
\sum_{a \in \A} \P(X=a) \log{\frac{\P(X=a)}{\Q(X=a)}} \leq \sum_{a \in \A} \frac{\big(\P(X=a)-Q(X=a)\big)^2}{\Q(X=a)},
\end{equation*}
where we define $\log 0/0 = 0$ for the left hand side and $0/0 = 0$ for the right hand side.
\begin{proof}
Using the concavity of the logarithm, then adding and subtracting $\Q(X=a)$ and expanding the square, and finally using the inequality $\log (1+u)\leq u$ for $u\geq 0$, we get
\begin{align*}
&\sum_{a\, \in\, \A} \P(X=a) \log{\frac{\P(X=a)}{\Q(X=a)}} \leq  \log{\left(\;\sum_{a\, \in\, \A}\frac{\P(X=a)^2}{\Q(X=a)}\right)}\\
& = \log{ \left( 1 + \sum_{a\, \in\, \A}\frac{\big(\P(X=a)-\Q(X=a)\big)^2}{\Q(X=a)}\right)}\\
& \leq \sum_{a\, \in\, \A}\frac{\big(\P(X=a)-\Q(X=a)\big)^2}{\Q(X=a)},
\end{align*}
which proves the lemma.
\end{proof}
\end{lem}

\noindent{\bf Acknowledgements.}
The authors acknowledge FAPESP (Regular Research Grants 2019/23439-4 and scholarship abroad 2024/06341-9) for support. 
D. Y. T. and S. G. gratefully acknowledge \'Ecole Polytechnique for supporting their visits to CPHT, funding a one-month stay in 2022 and another in 2024. D. Y. T. was partially supported by CNPq Grant 421955/2023-6.

\bibliographystyle{plain} 

\bibliography{guessing-biblio}       

\end{document}